\documentclass[a4paper,11pt,twoside]{amsart}

\usepackage[utf8]{inputenc}
\usepackage{amsmath, amsfonts, amssymb,amsthm}
\usepackage[mathscr]{eucal}
\usepackage[pdftex]{graphicx, color}

\definecolor{alert}{rgb}{0.8,0,0}

\usepackage[T1]{fontenc}
\usepackage[sc]{mathpazo}
\usepackage{enumerate}

\newcommand{\R}{\mathbb{R}}
\newcommand{\C}{\mathbb{C}}
\newcommand{\s}{\mathbb{S}}
\newcommand{\h}{\mathbb{H}}

\newcommand{\ads}{\mathbf{H}^3_1}

\newcommand{\pIm}{\mathrm{Im}}
\newcommand{\df}{\mathrm{d}}

\newcommand{\bz}{\bar{z}}

\renewcommand{\Re}{\mathrm{Re}}
\renewcommand{\Im}{\mathrm{Im}}

\newcommand{\prodesc}[2]{\left\langle #1, #2 \right \rangle}
\newcommand{\vprodesc}[2]{\langle\!\langle #1, #2 \rangle\!\rangle}
\newcommand{\abs}[1]{\left\lvert #1 \right\rvert}

\DeclareMathOperator{\cotanh}{cotanh}
\DeclareMathOperator{\arccotanh}{arccotanh}
\DeclareMathOperator{\sn}{sn} 
\DeclareMathOperator{\arcsn}{arcsn} 
\DeclareMathOperator{\cn}{cn}	
\DeclareMathOperator{\arccn}{arccn}	
\DeclareMathOperator{\tn}{tn}	
\DeclareMathOperator{\arctn}{arctn}	
\DeclareMathOperator{\dn}{dn}	
\DeclareMathOperator{\am}{am}	

\newtheorem{theorem}{Theorem}
\newtheorem*{theorem*}{Theorem}
\newtheorem{proposition}{Proposition}
\newtheorem{corollary}{Corollary}
\newtheorem{lemma}{Lemma}

\theoremstyle{definition}

\theoremstyle{remark}
  \newtheorem{remark}{Remark}

\numberwithin{equation}{section}

\title[Geometrical correspondence between minimal surfaces in $\ads$ and $\h^2\times \R$]{A geometrical correspondence between maximal surfaces in anti-De Sitter space-time and minimal surfaces in $\h^2\times \R$}

\author{Francisco Torralbo}
\address{Departement Wiskunde. KU Leuven. Celestijnenlaan 200B, B-3001 Leuven, Belgium}
\email{francisco.torralbo@wis.kuleuven.be}
\thanks{Research partially supported by a MCyT-Feder research project MTM2011-22547, Junta Andalucía Grants P09-FQM-5088 and P09-FQM-4496 and Belgian Interuniversity Attraction Pole P07/18 (Dygest)}



\begin{document}

\begin{abstract}
A geometrical correspondence between maximal surfaces in anti-De Sitter space-time and minimal surfaces in the Riemannian product of the hyperbolic plane and the real line is established. New examples of maximal surfaces in anti-De Sitter space-time are obtained in order to illustrate this correspondence.
\end{abstract}

\maketitle

\section{Introduction}\label{sec:introduction}

The study of minimal surfaces in product spaces $M^2\times \R$ was initiated by Rosenberg and Meeks~\cite{Rosenberg2002,MR2005} and has been very active since then. Among that spaces, there are three homogeneous Riemannian manifolds: $\R^3$, where the classical theory of minimal surfaces has been developed, and $\s^2\times \R$ and $\h^2\times \R$, where many authors have been actively working. Giving a complete list of references in the subject is far from being possible so we will only mention a few of them: Nelli and Rosenberg~\cite{NR2002} proved a Jenkins-Serrin-type theorem in $\h^2\times \R$, Hauswirth~\cite{Hauswirth2006} constructed minimal examples or Riemann type, Sá Earp and Tobiana~\cite{ST2004} investigated the screw motion invariant surfaces in $\h^2\times \R$, Daniel~\cite{Daniel2009} and Hauwirth, Sá Earp and Tobiana~\cite{HST2008} showed, independently, the existence of an associated family of minimal immersions for simply connected minimal surfaces in $\s^2\times \R$ and $\h^2\times \R$, Urbano and the author~\cite{TU2013} tackled a general study of minimal surfaces in $\s^2\times \s^2$ with applications to $\s^2\times \R$, and very recently Manzano, Plehnert and the author~\cite{MPT2013} constructed orientable and non-orientable even Euler characteristic embedded minimal surfaces in the quotient $\s^2\times \s^1$ and Martín, Mazzeo and Rogríguez~\cite{MMR2014} constructed the first examples of complete, properly embedded minimal surfaces in $\h^2\times \R$ with finite total curvature and positive genus.

In this paper we are going to show a geometric relation between maximal surfaces in anti-De Sitter space-time $\ads$ and minimal immersions in $\h^2\times \R$. It is well-known that the Gauss map of a spacelike maximal immersion in $\ads$ is always a minimal Lagrangian immersion in $\h^2\times \h^2$ (see~\cite{Torralbo2007} and also~\cite{CU2007} where an analogous case for Lagrangian minimal immersions in $\s^2\times \s^2$ is studied). We are going to get new minimal immersions in $\h^2\times\h^2$ by pairing different components of the Gauss map of two suitable maximal immersions in anti-De Sitter space-time (see Theorem~\ref{thm:Gauss-map-pair}). This construction is in the same spirit as in~\cite{TU2013}. From that, and under an appropriate choice of one element in the pair, we will establish a conformal correspondence between maximal immersions in anti-De Sitter space-time and minimal immersions in $\h^2\times\R$ (see Corollary~\ref{cor:Gauss-map-H2xR}). We will also show that this result admits a local converse, i.e.\ that roughly speaking every minimal surface in $\h^2\times \R$ is locally the Gauss map of a maximal immersion in anti-De Sitter space (see Theorem~\ref{thm:local-correspondence}).

Finally, we will illustrate this geometric correspondence by showing new examples of maximal surfaces in anti-De Sitter space in Proposition~\ref{prop:examples-AdS} and computing their Gauss map (in the sense of Corollary~\ref{cor:Gauss-map-H2xR}). This will provide us with two $1$-parameter families of minimal examples in $\h^2\times\R$. We want to point out that, although the constructed maximal immersions in $\ads$ are non-complete (see Proposition~\ref{prop:examples-AdS}), their corresponding minimal immersions in $\h^2\times \R$ induce complete metrics in the surface. Moreover, they are invariant by screw motions and they were first described in~\cite{ST2004}.

The structure of the paper is the following: Section~\ref{sec:preliminaries} introduces both anti-De Sitter space-time $\ads$ and the Riemannian products $\h^2\times\h^2$ and $\h^2\times \R$, where $\h^2$ stands for the hyperbolic plane. In Section~\ref{sec:maximal-AdS}, we will briefly present some basic facts about maximal surfaces in anti-De Sitter space as well as some examples. Section~\ref{sec:Gauss-map-of-a-pair} contains the main theorems that are illustrated in Section~\ref{sec:examples}. Finally, Section~\ref{sec:appendix} contains an analysis of the solutions to the sinh-Gordon equation that only depend on one variable.

\section{Preliminaries}\label{sec:preliminaries}

\subsection{The hyperbolic plane and the product manifold $\h^2 \times \h^2$}\label{subsec:hyperbolic-plane}

Let $\h^2$ be the hyperbolic plane and $\prodesc{\,}{\,}$ its metric. Although for all computations we will use the hyperboloid model of $\h^2$, i.e.\ $\h^2 = \{p \in \R^3:\, \prodesc{p}{p} = -1,\, p_1 > 0\}$, where $\prodesc{x}{y} = -x_1y_1 + x_2y_2 + x_3y_3$, the Poincaré disc model is also considered in Figures~\ref{fig:positive-energy} and~\ref{fig:negative-energy}. 

We endow $\h^2 \times \h^2$ with the product metric, also denoted by $\prodesc{\,}{\,}$. So $\h^2 \times \h^2$ is an Einstein manifold with constant scalar curvature $-4$.

We will consider $\h^2 \times \R$ as the totally geodesic submanifold of $\h^2 \times \h^2$ given by the image of the map $i:\h^2 \times \R \rightarrow \h^2 \times \h^2$ defined by $i(p,t) = [p, (\sinh(t), 0, \cosh(t))]$.

\subsection{The anti-De Sitter $3$-space}\label{subsec:anti-De-Sitter}
The anti-De Sitter $3$-space, that it is usually denoted by $\ads$, is a Lorentz manifold of dimension $3$ and constant curvature $-1$. It is defined as a quadric in a vector space. More precisely, let $\R^4_2$ be the euclidean $4$-space endow with the metric $\vprodesc{u}{v} = u_1v_1 + u_2v_2 - u_3v_3 - u_4v_4$. Then $\ads = \{p \in \R^4_2:\, \vprodesc{p}{p} = -1\}$. Moreover, the map $\pi: \ads \subset \R^4_2 \equiv \C^2_1 \rightarrow \h^2(-2)$, where $\h^2(c)$ stands for the hyperbolic plane of constant curvature $c < 0$, given by
\[
\pi(z, w) = \left(z\bar{w}, \frac{1}{2}(\abs{z}^2 + \abs{w}^2) \right),
\]
is a semi-Riemannian submersion with totally geodesic fibers generated by the unit temporal vector field $\xi_{(z,w)} = (iz, iw)$. The fiber of a point $(z_0,w_0) \in \ads$ is the circle $(z_0e^{it}, w_0e^{it})$. 

\subsection{The Gauss map}\label{subsec:definition-gauss-map}
Let $\phi: \Sigma \rightarrow \R^4_2$ be a spacelike immersion of an oriented surface $\Sigma$. Its Gauss map assigns to each point of the surface its oriented tangent plane. In this particular case, the image of the Gauss map is contained in $G^+_*(2,4)$, the Grassmann manifold of oriented spacelike planes of $\R^4_2$. It is well-known that $G^+_*(2,4)$ is diffeomorphic to $\h^2 \times \h^2$ (see, for example, \cite[Section 1]{Palmer1991}). To understand the construction in Section~\ref{sec:Gauss-map-of-a-pair} a suitable identification between $\h^2 \times \h^2$ and $G^+_*(2,4)$ is needed.

Let $\Lambda^2\R^4_2 = \mathrm{span}\{v \wedge w:\, v, w\in \R^4_2\} \equiv \R^6$ be the linear space generated by the $2$-vectors in $\R^4$ endowed with the index $2$ metric
\[
g(v \wedge w, v' \wedge, w') = \vprodesc{v}{w'}\vprodesc{w}{v'} - \vprodesc{v}{v'}\vprodesc{w}{w'}.
\]
The star operator $\star: \Lambda^2 \R^4_2\rightarrow \Lambda^2\R^4_2$ defined by $\alpha\wedge (\star \beta) = g(\alpha, \beta)\Omega$ for all $\alpha, \beta \in \Lambda^2 \R^4_2$, where $\Omega$ is the orientation form of $\R^4$, is an linear automorphism of $\Lambda^2 \R^4_2$ with eigenvalues $\pm 1$. Consider $\Lambda^2_\pm \R^4_2$ the eigenspaces of $\star$ associated to the eigenvalues $\pm 1$. Observe that we can decompose $\Lambda^2 \R^4_2 = \Lambda^2_+ \R^4_2 \oplus \Lambda^2_-\R^4_2$. Let $\{e_1, e_2, e_3, e_4\}$ be an oriented orthonormal frame of $\R^4_2$, i.e.\ $\abs{e_1}^2 = \abs{e_2}^2 = -\abs{e_3}^2 = -\abs{e_4}^2 = 1$ and $\vprodesc{e_i}{e_j} = 0$, $i \neq j$. The frame $\{E^\pm_j:\, j = 1, 2, 3\}$ given by:
\[
\begin{split}
E^\pm_1 &= \tfrac{1}{\sqrt{2}}(e_1 \wedge e_2 \pm e_4 \wedge e_3),\, 
E^\pm_2 = \tfrac{1}{\sqrt{2}}(e_1 \wedge e_3 \pm e_4 \wedge e_2),\,
E^\pm_3 = \tfrac{1}{\sqrt{2}}(e_1 \wedge e_4 \pm e_2 \wedge e_3),
\end{split}
\]
is an orthonormal oriented reference in $\Lambda^2_\pm \R^4_2$, i.e.\ $g(E^\pm_i, E^\pm_j) = \epsilon_i\delta_{ij}$, where $\epsilon_1 = -1$ and $\epsilon_2 = \epsilon_3 = 1$. Hence each $\Lambda^2_\pm \R^4_2$ is isometric to the Lorentz-Minkowski $3$-space. We denote by $\h^2_\pm$ the hyperbolic plane in the $3$-space $\Lambda^2_\pm \R^4_2$.

Finally, if $\{v, w\}$ is an oriented orthonormal frame of a plane $P \in G^+_*(2,4)$, then the map $G^+_*(2,4) \rightarrow \h^2_+ \times \h^2_-$ given by 
\[
P \mapsto \frac{1}{\sqrt{2}} [ v \wedge w + \star(v \wedge w), v\wedge w - \star(v \wedge w)]
\] 
is a diffeomorphism.

\section{Maximal surfaces in anti-De Sitter space}
\label{sec:maximal-AdS}

Let $\phi: \Sigma \rightarrow \ads$ be a spacelike maximal immersion, i.e.\ with zero mean curvature, of an oriented surface $\Sigma$ and $N$ a unit normal vector field to $\phi$. Given a conformal parameter $z = x+iy$ on $\Sigma$, it is well-known (see for instance~\cite{Palmer1990}) that the $2$-differential $\Theta_\phi(z) = \theta(z) \df z \otimes \df z= \vprodesc{\phi_z}{N_z}\df z\otimes \df z$ is holomorphic, where $N$ is the (timelike) unit normal vector field to $\phi$ such that $\{\phi_x, \phi_y, \phi, N\}$ is a positively oriented frame in $\R^4_2$ (we are using subscripts to indicate derivatives). The associated conformal factor $e^{2v}$ satisfies $v_{z\bar{z}} + e^{-2v}\abs{\Theta_\phi}^2 - \frac{1}{4}e^{2v} = 0$.  Moreover, the Frenet equations of the immersion are given by
\begin{equation}\label{eq:frenet-maxima-AdS}
\phi_{zz} = 2v_z\phi_z + \theta N, \quad \phi_{z\bz} = \tfrac{1}{2}e^{2v}\phi, \quad N_z = 2e^{-2v}\theta \phi_{\bz}.
\end{equation}

Conversely, we get the following result (see~\cite[Proposition 2.1]{Palmer1990} and also~\cite[Lemma~3.3]{Perdomo2009}):
\begin{quote}
\itshape
For any solution $v:D \rightarrow \R$, $D \subset \C$, to the equation $v_{z\bar{z}} - \frac{1}{2}\sinh(2v) = 0$ there exists a $1$-parameter family $\phi_t:\C \rightarrow \ads$ of maximal immersions whose induced metric is $e^{2v}\abs{\df z}^2$ and whose Hopf differential is $\Theta_{\phi_t}(z) = \frac{i}{2}e^{it}\df z \otimes \df z$.
\end{quote}

In this section we are going to present some examples of spacelike maximal surfaces in $\ads$ that will be useful in the sequel. The first simple example is the totally geodesic embedding of the hyperbolic plane $\h^2$ into $\ads$ given by
$
\mathbf{B} = \{(z, w) \in \R^4_2 \equiv \C^2:\, \pIm(w) = 0\}
$
, up to isometries of $\ads$.

The second example, that will play an important role in the following section (see Corollary~\ref{cor:Gauss-map-H2xR}), is the so-called \emph{hyperbolic cylinder}
\[
\mathbf{C} = \{(z, w) \in \R^4_2 \equiv \C^2:\, \Re(z)^2 - \Re(w)^2 = \Im(z)^2 - \Im(w)^2 = -\tfrac{1}{2}\}.
\]
It is a complete spacelike maximal surface with vanishing Gauss curvature, constant principal curvatures $\lambda_1 = -\lambda_2 = 1$ and the norm of the second fundamental form is $\abs{\sigma}^2 = 2$. It was characterized by Ishihara~\cite{Ishihara88} as the only complete maximal surface in $\ads$, up to rigid motions, with $\abs{\sigma}^2 = 2$. We can parametrize the hyperbolic cylinder $\mathbf{C}$ by 
\begin{equation}\label{eq:immersion-hyperbolic-cylinder}
\psi_t(x, y) = \frac{1}{\sqrt{2}} (\sinh a_t(x, y), \sinh b_t(x, y), \cosh a_t(x, y), \cosh b_t(x, y) ),
\end{equation}
where 
\[
\begin{split}
a_t(x, y) &= (x+y)\cos \tfrac{t}{2} + (x-y)\sin \tfrac{t}{2}, \\
b_t(x, y) &= (y-x)\cos \tfrac{t}{2} + (x+y)\sin \tfrac{t}{2}.
\end{split}
\]
Then $\psi_t(\R^2) = \mathbf{C}$, $z = x+iy$ is a conformal parameter with conformal factor $e^{2u(x, y)}$ where $u(x, y) = 0$, and the associated Hopf differential is $\Theta_{\psi_t} = \frac{i}{2}e^{it}\df z \otimes \df z$.

Next we are going to show examples invariant under a $1$-parameter group of isometries of $\ads$. In that case, the equation for its conformal factor $v_{z\bar{z}} - \frac{1}{2}\sinh(2v) = 0$ becomes an ordinary differential equation 
\begin{equation}\label{eq:sinh-Gordon-ordinary}
v''(x) -2\sinh(2v(x)) = 0, \text{ with energy } E =  \frac{1}{2}v'(x)^2 - \cosh( 2v(x) ).
\end{equation}
It is possible to integrate explicitly the Frenet system~\eqref{eq:frenet-maxima-AdS} for some values of $E$ obtaining the following result (see also Section~\ref{sec:appendix} where we get all the solutions to the previous equation in terms of Jacobi elliptic functions).

\begin{proposition}\label{prop:examples-AdS}
Let $v: I \subseteq \R \rightarrow \R$ be a solution to~\eqref{eq:sinh-Gordon-ordinary} with energy $E$. Then the map $\phi_E: \Sigma_v = (I \times \R, e^{2v}g_0) \rightarrow \ads$ given by:
\[
\begin{split}
\phi_E&(x, y) = \frac{1}{\sqrt{2E}}\left(e^{v(x)} \cos(\sqrt{2E} y), -e^{v(x)} \sin(\sqrt{2E} y), \right.\\
& \left.-\sqrt{2E + e^{2v(x)}} \cos(\sqrt{2E} G(x)), -\sqrt{2E + e^{2v(x)}} \sin(\sqrt{2E} G(x))\right), \, E > 0, \\
\phi_E&(x,y) = \frac{1}{\sqrt{-2E}}\left( \sqrt{-2E - e^{2v(x)}} \sinh(\sqrt{-2E}G(x)), e^{v(x)}\sinh(\sqrt{-2E}y)\right. \\
	&\left. e^{v(x)}\cosh(\sqrt{-2E}y), \sqrt{-2E - e^{2v(x)}} \cosh(\sqrt{-2E}G(x))\right),\, E < 0, 
\end{split}
\]
is an isometric maximal immersion with associated Hopf differential $\Theta(z) = \frac{1}{2}\df z\otimes \df z$, where $G(x) = \int_0^x \frac{\df t}{2E + e^{2v(t)}}$ and $g_0$ is the Euclidean metric in $\R^2$.

Moreover, all the surfaces $\Sigma_v$, except the one associated to the trivial solution $v(x) = 0$ which is the hyperbolic cylinder, are not complete.
\end{proposition}

\begin{remark}\label{rmk:examples-maximal-AdS}
The obtained examples in Proposition~\ref{prop:examples-AdS} are invariant by the following $1$-parameter group of isometries depending on the sign of $E$:

\begin{center}
\begin{tabular}{cc}
	$E > 0$ &  $E < 0$ \\ 
	$\begin{pmatrix}
			\cos \theta 	&	-\sin \theta 	&	0 	& 	0 \\
			\sin \theta 	&	\cos \theta 	&	0 	&	0 \\
			0 	&	0 	&	1 	&	0 \\
			0 	& 	0 	& 	0 	& 	1
		\end{pmatrix}$
	& 
	$\begin{pmatrix}
			1 	&	 0 	& 	0 	& 	0 	\\
			0 	&	\cosh \theta 	&	\sinh \theta 	& 	0 \\
			0 	&	\sinh \theta 	&	\cosh \theta 	&	0 \\
			0 	& 	0 	& 	0 	& 	1
		\end{pmatrix}$
\end{tabular}
\end{center}

Besides, when $E < 0$ the immersion $\phi_E$ is only defined for $2E < -e^{2v(x)}$. A deep analysis of the solutions (see Section~\ref{sec:appendix}) shows that there is always a non-empty interval $I' \subseteq I$ where this happens, being $I$ the maximal interval of definition of $v$ (see also the proof of Proposition~\ref{prop:examples-AdS}). 
\end{remark}

\begin{remark}
It is not surprising that the obtained examples are not complete. In~\cite{Perdomo2009}, the author pointed out the difficulties of getting complete maximal immersions in anti-De Sitter space-time and he also provided with complete examples looking for radial solutions of the sinh-Gordon equation. Palmer in~\cite[Theorem I]{Palmer1990} also provide complete minimal examples in $\ads$ in terms of holomorphic quadratic differentials.
\end{remark}

\begin{proof}
Taking equation~\eqref{eq:sinh-Gordon-ordinary} into account, it is straightforward to check that $\phi_E$ is a maximal isometric immersion. We will now show that the metric $e^{2v}g_0$ is non-complete except for the trivial solution $v = 0$., by finding a divergent curve $\gamma$ in $\Sigma$ with finite length. Notice that we can consider the solutions to~\eqref{eq:sinh-Gordon-ordinary} given in Lemma~\ref{lm:solutions-sinh-Gordon-ordinary} for $a_0 = 0$, which are always defined in an interval $I = ]0, \ell[$. Observe that if $v$ is a solution then $-v$ is also a solution (see Section~\ref{sec:appendix}) so we have to deal with both cases.

If $E > -1$, thanks to the symmetries of the solutions (see Remark~\ref{rmk:properties-solutions-sinh-Gordon}.(1)), $v$ and $-v$ are symmetric so, considering the one given in Lemma~\ref{lm:solutions-sinh-Gordon-ordinary} we have that $e^{v(x)} \leq 1$ in $]0,\ell/2[$. Hence the curve $\gamma:]0, \tfrac{\ell}{2}[ \rightarrow \Sigma$, given by $\gamma(t) = \tfrac{\ell}{2} - t$, diverges in $\Sigma$ but has finite length.

If $E = -1$ we have three different solutions: (1) $v(x) = 0$, which produces the hyperbolic cylinder which is complete; (2) $v(x) = \log \tanh(x)$, in this case $\Sigma = (\R^+\times \R, \tanh^2(x)g_0)$ and so the curve $\gamma(t) = (a-t,0)$, $t\in ]0,a[$, $a\in \R$ arbitrary, diverges in $\Sigma$ but has finite length; and (3) $v(x) = \log \cotanh(x)$. In this case the immersion is only defined when $\cotanh^2(x) < 2$ (see Remark~\ref{rmk:examples-maximal-AdS}), i.e. $\Sigma = (]0, \arccotanh(\sqrt{2})[ \times \R, \cotanh^2(x)g_0)$. Hence the curve $\gamma(t) = (t,0)$, $t\in]\tfrac{1}{2}, \arccotanh(\sqrt{2})[$ diverges in $\Sigma$ but has finite length.

Finally, if $E < -1$ we have two different types of solutions, namely $v_1(x) = -\log(\tfrac{1}{\lambda}\sn_\mu(\lambda x))$ and $v_2(x) = -v_1(x)$ (see Lemma~\ref{lm:solutions-sinh-Gordon-ordinary}). In the first case the immersion is only defined when $2E+e^{2v_1(x)} < 0$, i.e.\ for $x \in J = ]c, \ell - c[$ where $c = \tfrac{1}{\lambda}\arcsn_\mu(\tfrac{\lambda}{\sqrt{-2E}})$ (see Lemma~\ref{lm:solutions-sinh-Gordon-ordinary} for the definition of $\lambda$ and $\mu$). But $e^{v_1(x)}\leq \sqrt{-2E}$ for $x \in ]c, \ell -c[$ and so $\Sigma$ is also incomplete in this case.

In the second case, $2E +e^{2v_2(x)} < 0$ so $\Sigma = (]0, \ell[ \times \R, e^{2v_2(x)}g_0)$, but $e^{v_2(x)}\leq \tfrac{1}{\lambda}$ so the curve $\gamma(t) = (\ell/2 - t, 0)$, $t \in ]0, \ell/2[$, diverges in $\Sigma$ and has finite length.
\end{proof}

\section{The Gauss map of a pair of maximal surfaces in $\ads$}
\label{sec:Gauss-map-of-a-pair}

Let $\phi: \Sigma \rightarrow \ads \subset \R^4_2$ be a spacelike immersion of an oriented surface $\Sigma$. The \emph{Gauss map} of $\phi: \Sigma \rightarrow \R^4_2$ is the map $\nu_\phi = (\nu^+_\phi, \nu^-_\phi): \Sigma \rightarrow \h^2_+ \times \h^2_-$ defined by
\[
\nu^\pm_\phi(p) = \frac{1}{\sqrt{2}}[e_1 \wedge e_2 \pm N(p) \wedge \phi(p)],
\]
where $\{e_1, e_2\}$ is an oriented orthonormal basis in $T_p \Sigma$ and $N$ is the unit (timelike) normal vector field to the immersion $\phi$ such that $\{e_1, e_2, \phi(p), N_p\}$ is oriented in $\R^4_2$ (see Section~\ref{subsec:definition-gauss-map}). 

If $\phi$ is maximal then its Gauss map is a Lagrangian minimal immersion in $\h^2\times\h^2$ (see~\cite{Torralbo2007}). For instance: 
\begin{itemize}
	\item The Gauss map of the totally geodesic embedding of the hyperbolic plane in $\ads$ given in Section~\ref{sec:maximal-AdS} is the diagonal map $\nu: \h^2\to \h^2\times \h^2$, $\nu(p) = (p,p)$.
	\item The Gauss map of the hyperbolic cylinder is the product of two geodesics of $\h^2$.
\end{itemize}

\begin{theorem}\label{thm:Gauss-map-pair}
Let $\Sigma$ be a Riemann surface and $\phi$, $\psi: \Sigma \rightarrow \ads$ two conformal spacelike maximal immersions with the same Hopf differentials $\Theta_\phi = \Theta_\psi$. Then
\[
	\nu_{\{\phi, \psi\}}: (\nu^+_\phi, \nu^-_\psi): \Sigma \rightarrow \h^2 \times \h^2
\]
is a conformal minimal immersion. Moreover, the induced metric by $\nu_{\{\phi, \psi\}}$ is
\[
	g = \frac{1}{2}\Bigl[ (2 + \abs{\sigma_\phi}^2) g_\phi + (2 + \abs{\sigma_\psi}^2)g_\psi\Bigr],
\]
where $g_\phi$ and $g_\psi$ are the induced metrics on $\Sigma$ by $\phi$ and $\psi$, respectively. Here $|\sigma_\phi|$ and $|\sigma_\psi|$ are the lengths of the second fundamental forms of $\phi$ and $\psi$ in $\ads$, computed with respect to $g_\phi$ and $g_\psi$, respectively. 
\end{theorem}

\begin{remark}~
\begin{enumerate}
	\item If $\phi = \psi$, then $\nu_{\{\phi, \psi\}} = \nu_\phi$ is the Gauss map of $\phi$. 
	\item Given a maximal immersion $\phi: \Sigma \rightarrow \ads$, its \emph{polar} immersion (possibly branched) is $N: \Sigma \rightarrow \ads$, where $N$ is a unit normal vector field to $\phi$. $N$ is also a maximal conformal immersion with the same Hopf differential as $\phi$. Nevertheless, $\nu_{\{\phi, N\}}$ is congruent to $\nu_\phi$, the Gauss map of $\phi$.
	\item Given $A \in \mathcal{O}_2(4)$, then it is easy to check that $\nu_{\{A\phi, \psi\}}$ is congruent to $\nu_{\{\phi, \psi\}}$.
\end{enumerate}

\end{remark}

\begin{proof}
	The immersion $\phi: \Sigma \rightarrow  \R^4_2$ has parallel mean curvature vector because it is contained in $\ads$ as a maximal surface. From~\cite[Theorem 3.2]{Palmer1991} we deduce that each $\nu^\pm_\phi$ is a harmonic map. Analogously $\nu_\psi^\pm:\Sigma \rightarrow \h^2$ are also harmonic maps. Hence $\nu_{\{\phi, \psi\}} = (\nu_\phi^+, \nu_\psi^-):\Sigma \rightarrow \h^2\times \h^2$ is a harmonic map.

	It remains to check that $\nu$ is conformal (and so minimal). Let $z = x+iy$ a conformal parameter over $\Sigma$ and $N_\phi$, $N_\psi$ the temporal unit normal vector field to $\phi$ and $\psi$ respectively such that $\{\phi_x, \phi_y, \phi, N_\phi\}$ and $\{\psi_x, \psi_y, \psi, N_\psi\}$ are oriented references in $\R^4_2$. Since $\phi$ and $\psi$ are conformal immersions the induced metrics by $\phi$ and $\psi$ in $\Sigma$ are given by $g_\phi = e^{2u}\abs{\df z}^2$ and $g_\psi = e^{2w}\abs{\df z}^2$ for certain functions $u$ and $w$. Moreover, $\Theta_\phi = \Theta_\psi = \theta \df z \otimes \df z$ for some function $\theta(z)$ by hypothesis. 

	We can express the component of the Gauss map $\nu_{\{\phi, \psi\}}$ as
	\[
		\begin{split}
			v_\phi^+(z) = \tfrac{1}{\sqrt{2}}\bigl( -2ie^{-2u} \phi_z \wedge \phi_{\bz}  - \phi \wedge N_\phi \bigr), \\
			v_\psi^-(z) = \tfrac{1}{\sqrt{2}}\bigl( -2ie^{-2w} \psi_z \wedge \psi_{\bz}  + \psi \wedge N_\psi \bigr), \\
		\end{split}
	\]
	Taking the Frenet equations~\eqref{eq:frenet-maxima-AdS} of $\phi$ and $\psi$ and Section~\ref{subsec:definition-gauss-map} into account, we easily get that
	\[
	\begin{split}
		(\nu_\phi^+)_z &= \tfrac{1}{2}e^u(-i + 2\theta e^{-2u})E^+_2(z) + \tfrac{i}{2}e^u(i + 2\theta e^{-2u})E^+_3(z), \\
		(\nu_\psi^-)_z &= \tfrac{1}{2}e^w(-i - 2\theta e^{-2w})E^-_2(z) + \tfrac{i}{2}e^w(-i + 2\theta e^{-2w})E^-_3(z). \\
	\end{split}
	\]
	Then we deduce from the previous equations that:
	\begin{align*}
		\vprodesc{(\nu_\phi^+)_z}{(\nu_\phi^+)_z} = -2i\theta, 
		& & \vprodesc{(\nu_\psi^-)_z}{(\nu_\psi^-)_z} = 2i\theta, \\
		|\nu_\phi^+|^2 = \tfrac{1}{2}\bigl(e^{2u} + 4e^{-2u}\abs{\theta}^2 \bigr), 
		& & |\nu_\psi^-|^2 = \tfrac{1}{2}\bigl(e^{2w} + 4e^{-2w}\abs{\theta}^2 \bigr).
	\end{align*}
	Finally, $\vprodesc{\nu_z}{\nu_z} = 0$ and so $\nu = \nu_{\{\phi,\psi\}}$ is a conformal map. Moreover, from $8\abs{\theta}^2 = e^{4u}\abs{\sigma_\phi}^2 = e^{4w}\abs{\sigma_\psi}^2$ and 
	\[
		\abs{\nu_z}^2 = \tfrac{1}{2}\bigl(e^{2u} + e^{2w} + 4\abs{\theta}^2(e^{-2u} + e^{-2w}) \bigr),
	\]
	we get the expression of the induced metric on $\Sigma$ by $\nu$.
\end{proof}

Now, let $\phi: \Sigma \rightarrow \ads$ be a conformal maximal immersion. There is no loss of generality in assuming that locally the Hopf differential $\Theta(z) = \tfrac{i}{2}e^{it}\df z \otimes \df z$ (observe that either $\Theta = 0$ and so the immersion is totally geodesic or the zeroes of $\Theta$ are isolated and we can locally normalize $\Theta$ away from the zeroes). Then $\phi$ and $\psi_t$, the immersion of the hyperbolic cylinder given in Section~\ref{sec:maximal-AdS}, are two conformal maximal immersions with the same Hopf differentials. Then, thanks to the previous theorem, $\hat{\nu}_\phi = \nu_{\{\phi, \psi_t\}}: \Sigma \rightarrow \h^2\times\h^2$ is a minimal immersion that we call the \emph{modified Gauss map} of $\phi$. Now, the Gauss map of the hyperbolic cylinder $\psi_t$ is the product of two geodesics in $\h^2\times\h^2$ so its second component can be viewed as a map from $\Sigma$ to $\R$. Hence the \emph{modified Gauss map of a maximal immersion $\phi:\Sigma \rightarrow \ads$ is a conformal minimal immersion $\hat{\nu}_\phi: \Sigma \rightarrow \h^2\times \R$}. We get the following result:

\begin{corollary}\label{cor:Gauss-map-H2xR}
	Let $u:D \subseteq \C \rightarrow \R$ be a solution of $u_{z\bz} - \tfrac{1}{2}\sinh(2u) = 0$. Then the $1$-parameter family of minimal immersions $\Phi_t: (D, 4\cosh^2 u \abs{\df z}^2) \rightarrow \h^2\times \R$ with Hopf differential $\Theta = e^{it}\df z\otimes \df z$ (see~\cite[Corollary~10]{HST2008}) associated to $u$ is given by:
	\[
		\Phi_t(z) = \bigl( \nu_{\phi_t}^+(z), 2\,\pIm(ze^{it/2})\bigr),
	\]
	where $\phi_t: (D, e^{2u}\abs{\df z}^2) \rightarrow \ads$ is the $1$-parameter family of immersions associated to $u$ with Hopf differential $\Theta_{\phi_t} = \tfrac{i}{2}e^{it}\df z \otimes \df z$.
\end{corollary}

\begin{proof}
	Let $\psi_t = (\C, \abs{\df z}^2) \rightarrow \ads$ the immersion of the hyperbolic cylinder given in equation~\eqref{eq:immersion-hyperbolic-cylinder}. Then $\phi_t, \psi_t:D \subseteq \C \rightarrow \ads$ are conformal maximal immersions with the same Hopf differentials. Hence, applying the previous theorem we get that $\nu_t = \nu_{\{\phi_t, \psi_t\}}$ is a conformal minimal immersion in $\h^2\times \h^2$ with induced metric $4\cosh^2 u\abs{\df z}^2$. 

	Furthermore, a straightforward computation shows that
	\[
	\nu_\psi^-(z) = \cosh[2\, \pIm(ze^{it/2})] E_1^- + \sinh[2\, \pIm(z e^{it/2}) ]E_3^-,
	\] 
	where $\{E_1^-, E_2^-, E_3^-\}$ is the orthonormal reference in $\Lambda^2_-\R^4_2$ associated with the canonical base of $\R^4_2$ (see Section~\ref{subsec:definition-gauss-map}). Hence, as we have mentioned above, $\nu_\psi^-(D)$ is contained in a geodesic of $\h^2$. Considering $\R$ embedded in $\h^2$ as such geodesic we get the result. Finally, it is easy to check that $\Theta_{\nu_t} = -2i \Theta_{\phi_t} = \Theta$.
\end{proof}

The next result is a local converse of Theorem~\ref{thm:Gauss-map-pair} in the special case of surfaces immersed in $\h^2\times \R$.

\begin{theorem}\label{thm:local-correspondence}
	Let $\phi: \Sigma \rightarrow \h^2\times \R$ an isometric minimal immersion of a simply connected Riemann surface $\Sigma$ satisfying $\nu^2 < 1$, where $\nu =\prodesc{N}{\partial_t}$. Then, there exists a conformal maximal immersion $\psi:\Sigma \rightarrow \ads$ such that $\phi = \hat{\nu}_{\psi}$, up to an ambient isometry.
\end{theorem}

\begin{proof}
Let $w$ a conformal parameter over $\Sigma$. Then $\Upsilon(w) = \prodesc{\phi_w}{\partial_t}\df w$ is a holomorphic  1-form without zeroes (note that $\abs{\Upsilon}^2 = \tfrac{1}{4}(1-\nu^2) > 0$ by assumption). Then we can always find another conformal parameter $z$ such that $\Upsilon = \df z$. The conformal factor induced by $\phi$ in this new parameter is $4\cosh^2 u$, where $u = \tanh(\nu)$ satisfies $u_{z\bz} - \tfrac{1}{2}\sinh(2u) = 0$. Moreover, the fundamental data of the immersion $\phi$ can be expressed in terms of $u$ (see~\cite[Theorem 2.3]{FM2010}).

Let $\psi: \Sigma \rightarrow \ads$ be the maximal conformal immersion associated to $u$ with Hopf differential $\Theta_\psi(z) = \frac{i}{2}\df z\otimes \df z$ (see Section~\ref{sec:maximal-AdS}). Then, $\hat{\nu}_\psi: \Sigma \rightarrow \h^2\times \R$ is a minimal immersion with the same fundamental data as $\phi$ and so both immersions differ in an ambient isometry.
\end{proof}

\begin{remark}
It is possible to get a similar result for minimal immersion of $\h^2\times \h^2$ without complex points as in~\cite[Theorem 3]{TU2013}, that is, every minimal immersion in $\h^2\times\h^2$ without complex points is locally congruent to the Gauss map of the pair of two maximal immersion in the anti-De Sitter space-time.
\end{remark}

\section{Examples}\label{sec:examples}

In this section we are going to use Corollary~\ref{cor:Gauss-map-H2xR} to compute the minimal immersions associated to the maximal immersions in $\ads$ given by Proposition~\ref{prop:examples-AdS}. As we shall see, the obtained examples are invariant by $1$-parameter groups of isometries of $\h^2\times\R$, namely, \emph{elliptic} and \emph{hyperbolic} screw motions (see figures~\ref{fig:positive-energy} and~\ref{fig:negative-energy}). Moreover,  although the considered maximal immersions in $\ads$ are not complete (see Remark~\ref{rmk:examples-maximal-AdS}), their Gauss maps, in the sense of Corollary~\ref{cor:Gauss-map-H2xR}, are complete immersions.

{\itshape
	Let $v: I \subseteq \R \rightarrow \R$ be a solution of $v''(x) - 2\sinh(2v) = 0$ with energy $E$ (cf.\ equation~\eqref{eq:sinh-Gordon-ordinary}). Then, the map $\Phi_E: \Sigma = (I \times \R, 4\cosh^2(v)g_0) \rightarrow \h^2\times\R$ given by:
	\[
	\begin{split}
		\Phi_E&(x,y) = \tfrac{1}{\sqrt{2E}}\left(v'(x), \tfrac{\sqrt{2E}e^{-v(x)}\cos \sqrt{2E}\bigl(y-G(x) \bigr) - v'(x)e^{v(x)}\sin \sqrt{2E}\bigl(y-G(x)\bigr)}{\sqrt{2E + e^{2v(x)}}}, \right. \\
		&\left. \tfrac{\sqrt{2E}e^{-v(x)}\sin \sqrt{2E}\bigl(y-G(x) \bigr) + v'(x)e^{v(x)}\cos \sqrt{2E}\bigl(y-G(x)\bigr)}{\sqrt{2E + e^{2v(x)}}}, 2\sqrt{E}(y-x) \right),\quad E > 0, \\
		\Phi_E&(x,y) = \tfrac{1}{\sqrt{-2E}}\left( \tfrac{\sqrt{-2E}e^{-v(x)} \cosh\sqrt{-2E}\bigl(y-G(x)\bigr) - v'(x)e^{v(x)}\sinh \sqrt{-2E}\bigl( y-G(x)\bigr)}{\sqrt{-2E-e^{2v(x)}}}, v'(x), \right. \\
		&\left. \tfrac{-\sqrt{-2E}e^{-v(x)} \sinh\sqrt{-2E}\bigl(y-G(x)\bigr) + v'(x)e^{v(x)}\cosh \sqrt{-2E}\bigl( y-G(x)\bigr)}{\sqrt{-2E-e^{2v(x)}}}, 2\sqrt{-E}(y-x)\right), \quad E < 0,
	\end{split}
	\]
	is an isometric minimal immersion with associated Hopf differential $\Theta = -i \df z\otimes \df z$, where $G(x) = \int_0^x \frac{\df t}{2E + e^{2v(t)}}$  and $g_0$ stands for the Euclidean metric in $\R^2$.
}

\begin{figure}[htbp]
\begin{tabular}{ccc}
\includegraphics[width=0.33\textwidth]{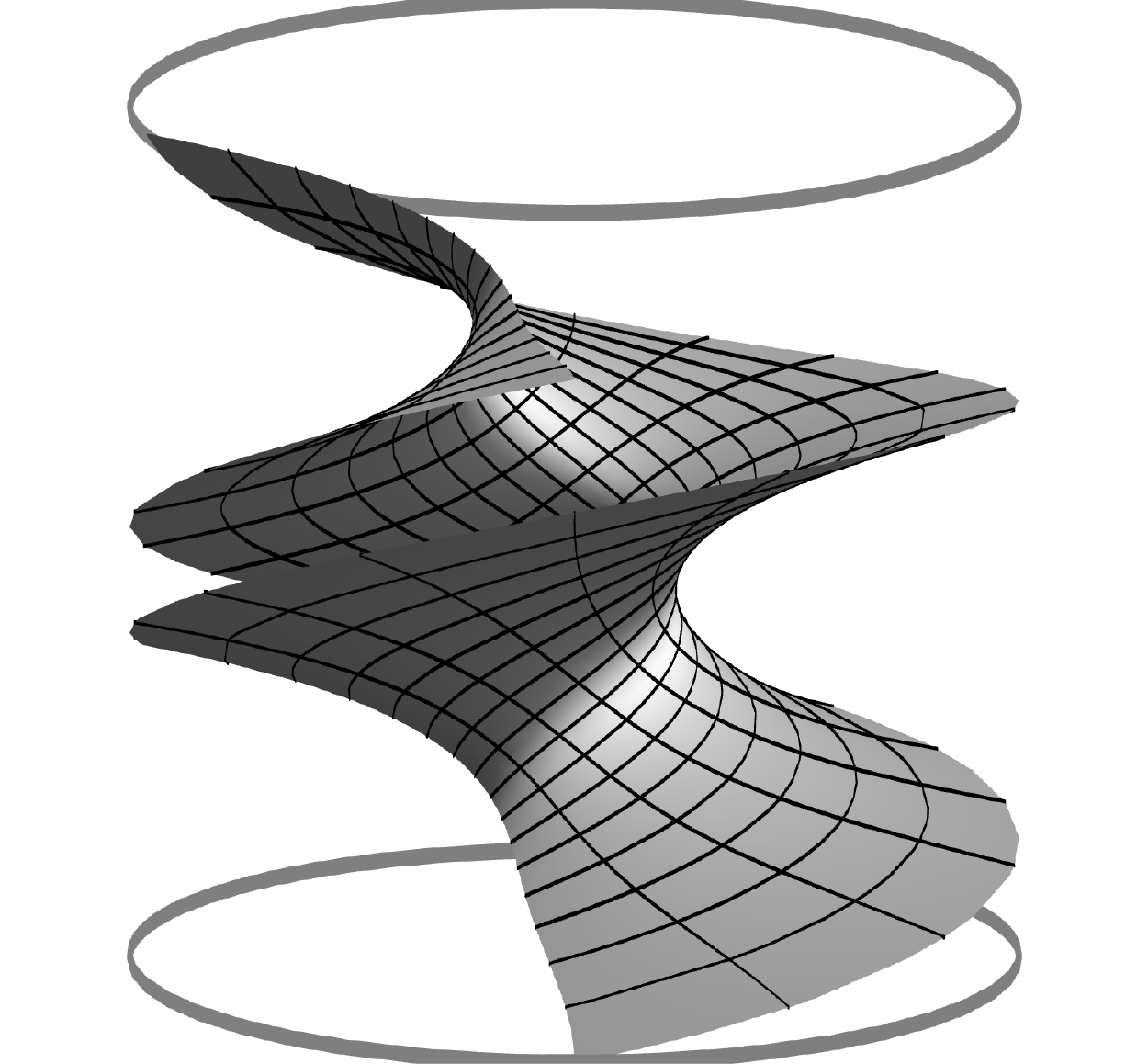}	&  \includegraphics[width=0.33\textwidth]{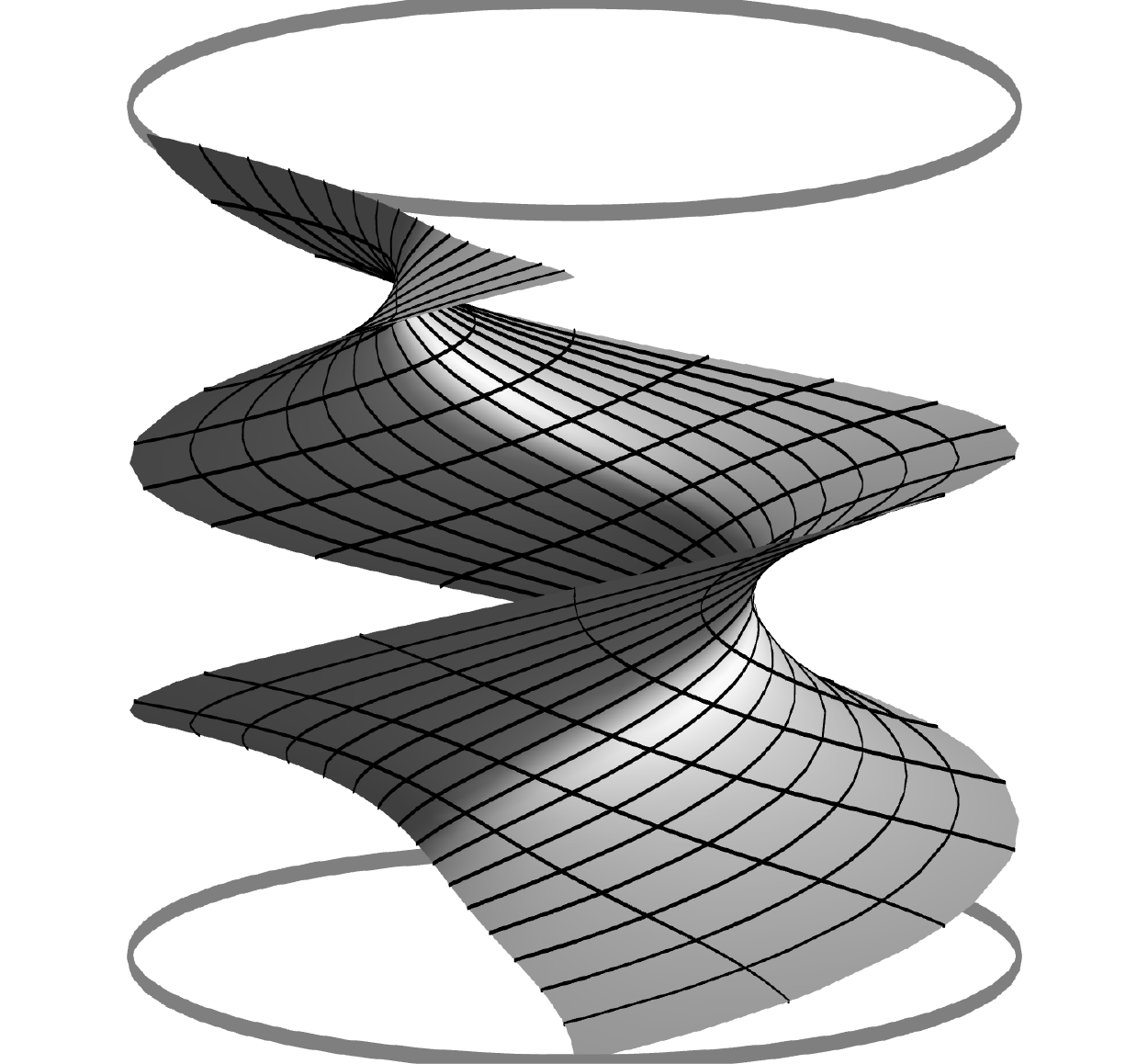} & \includegraphics[width=0.33\textwidth]{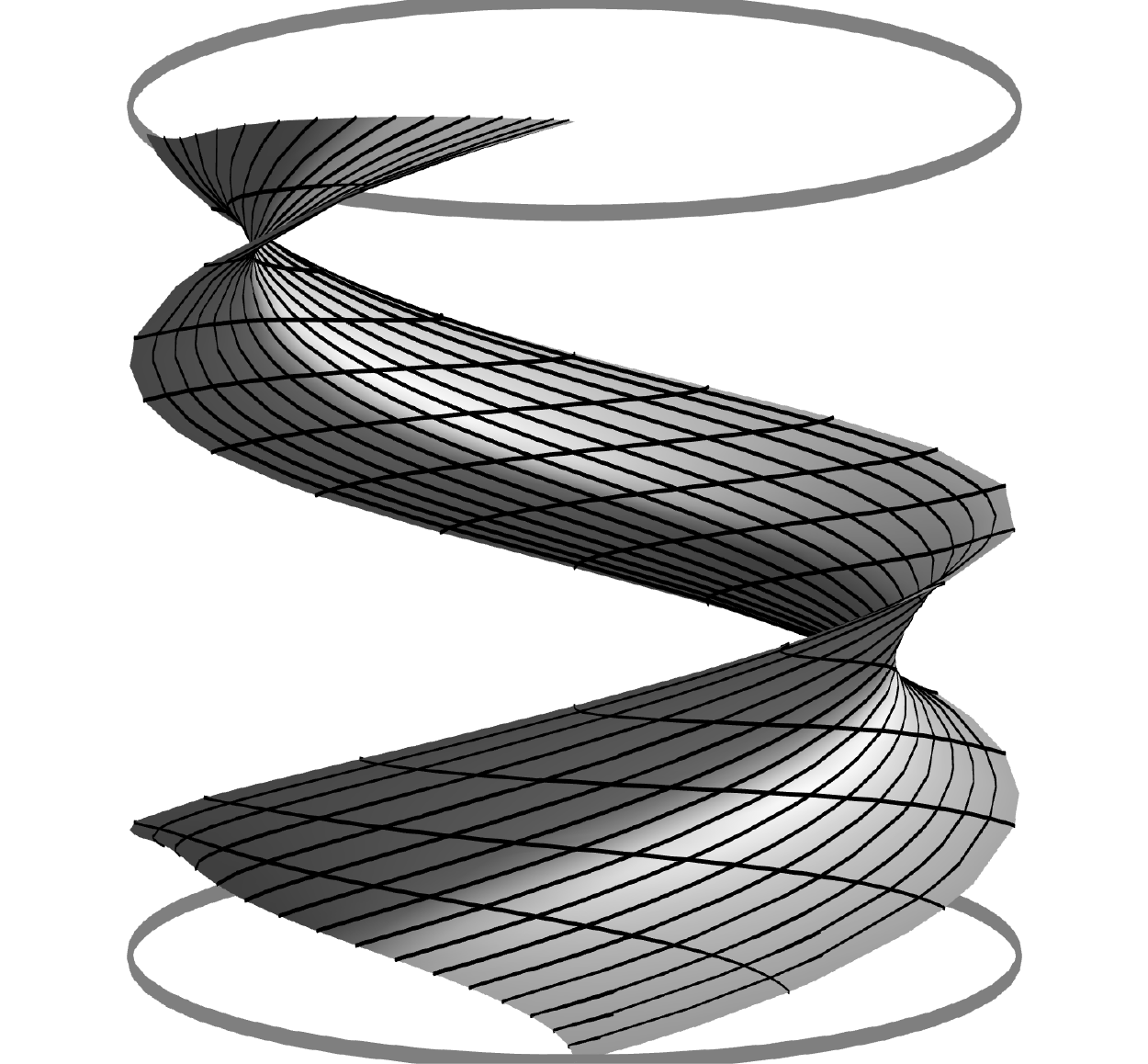} \\
\includegraphics[width=0.33\textwidth]{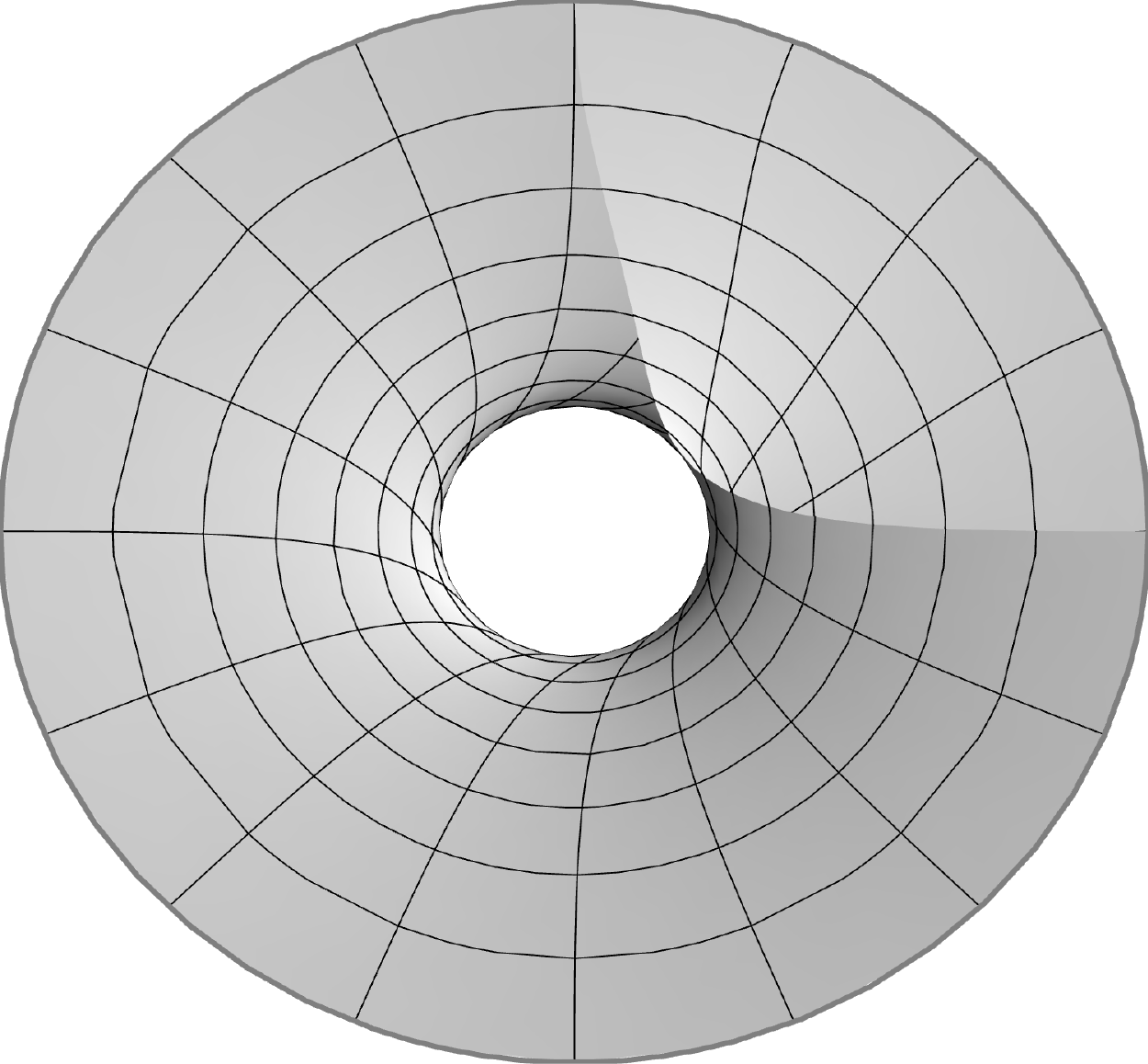}	&  \includegraphics[width=0.33\textwidth]{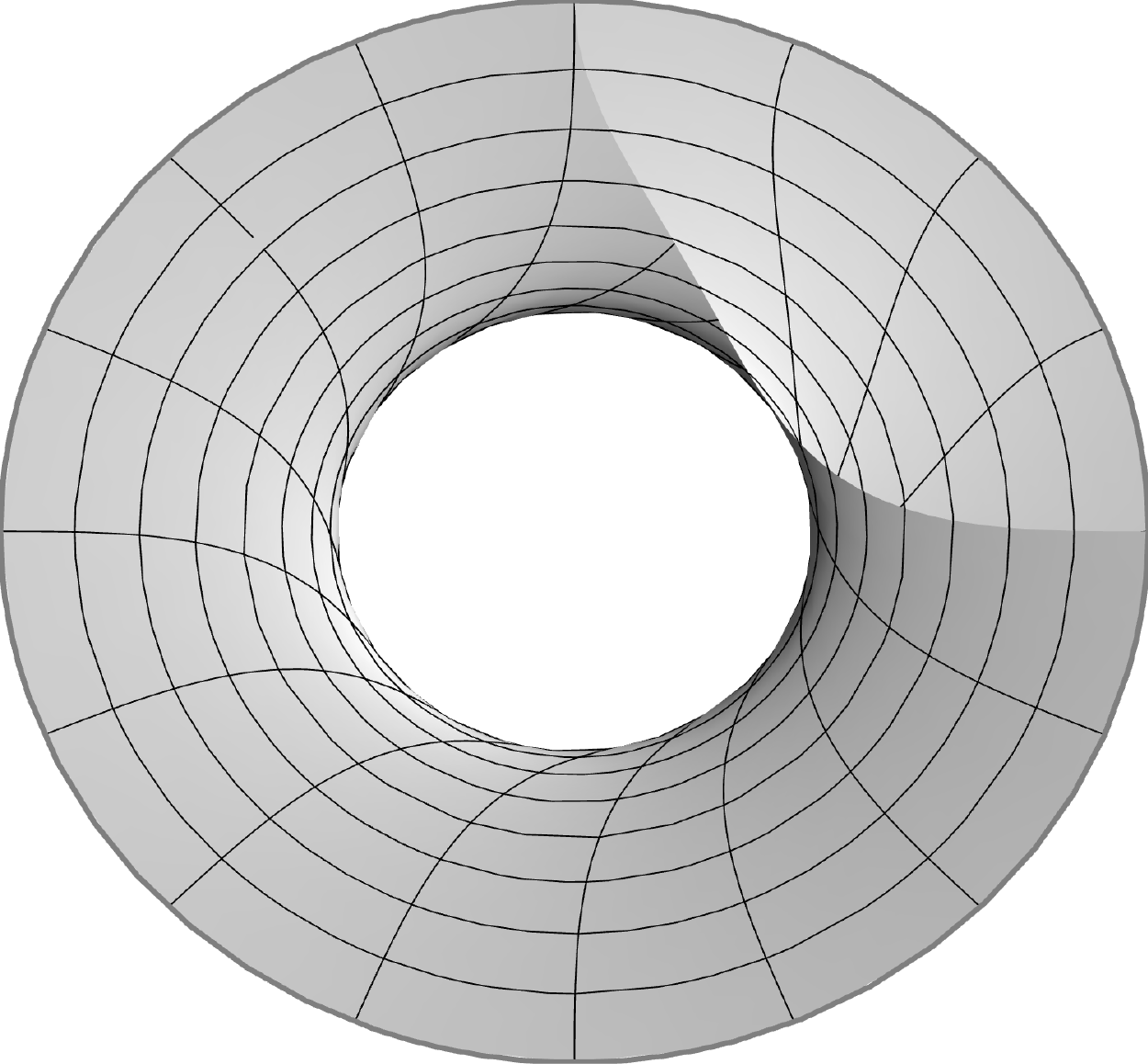} & \includegraphics[width=0.33\textwidth]{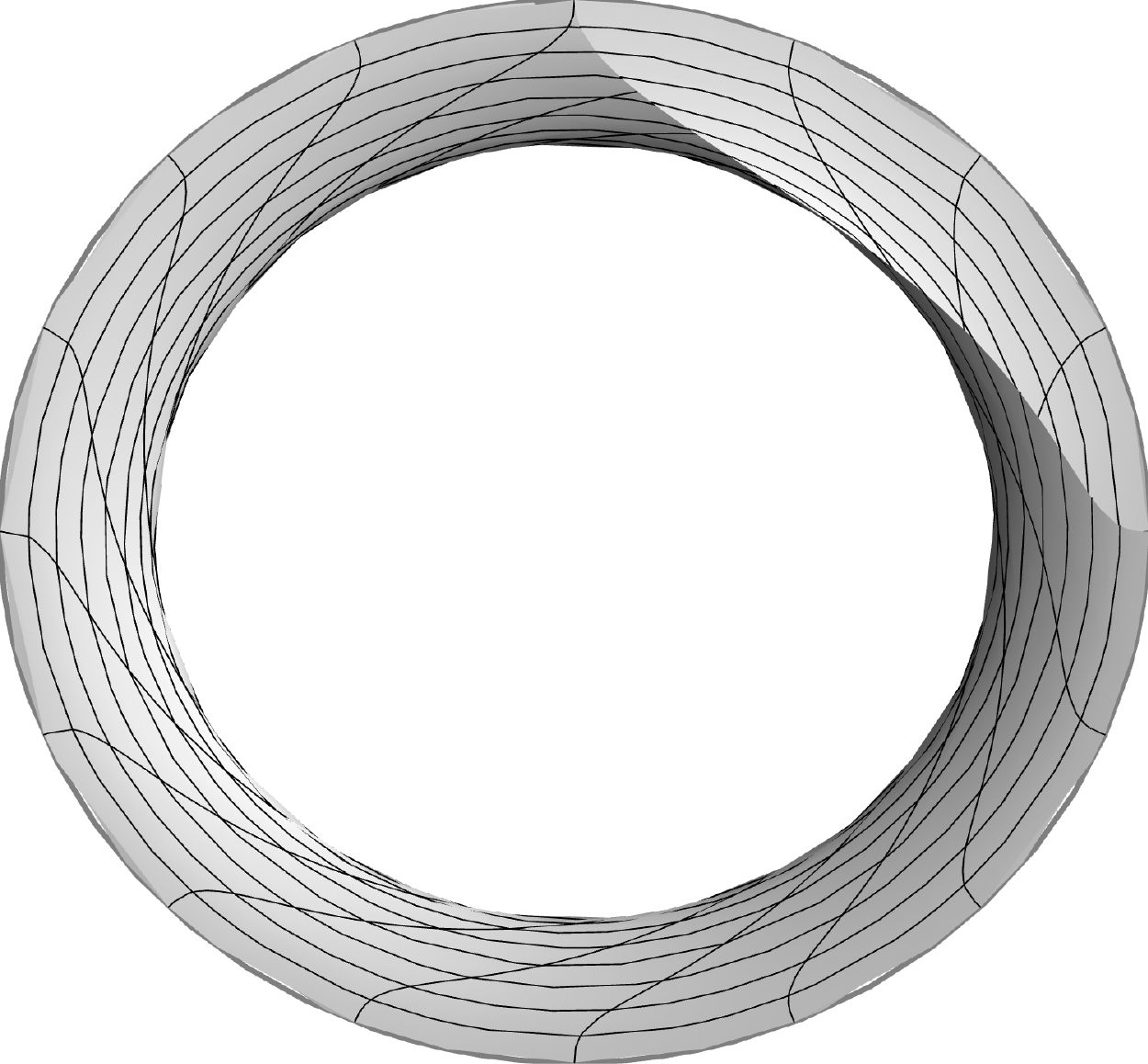}
\end{tabular}
\caption{From left to right, typical solutions for positive energy $E = 4, 1, 0.1$ in $\h^2\times \R$ being $\h^2$ the disc model. Below each surface the top view has been drawn. The boundary of $\h^2$ is drawn to help the visualization.} \label{fig:positive-energy}
\end{figure}

\begin{figure}[htbp]
\begin{tabular}{ccc}
\reflectbox{\includegraphics[width=0.33\textwidth]{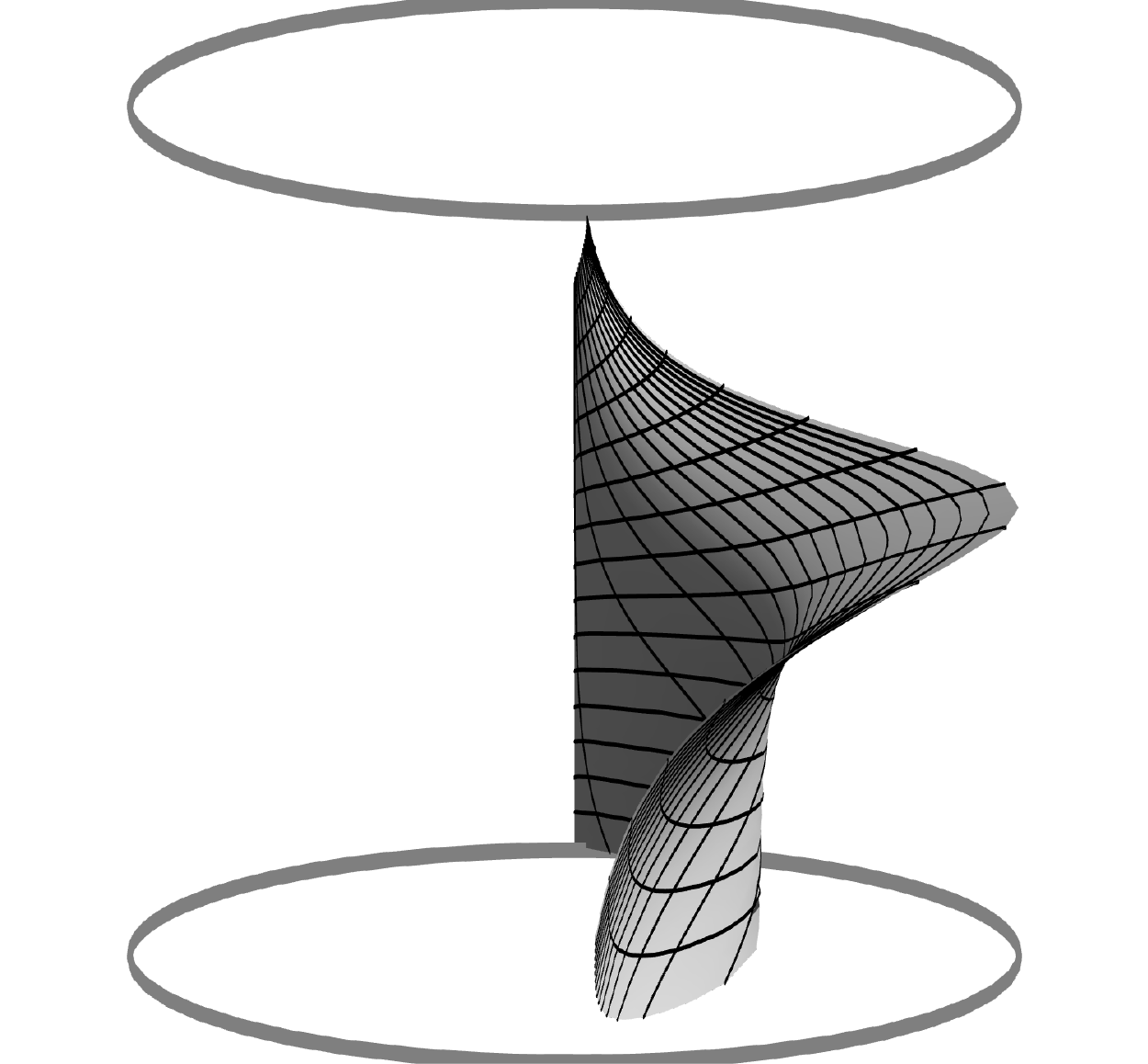}}	&  \includegraphics[width=0.33\textwidth]{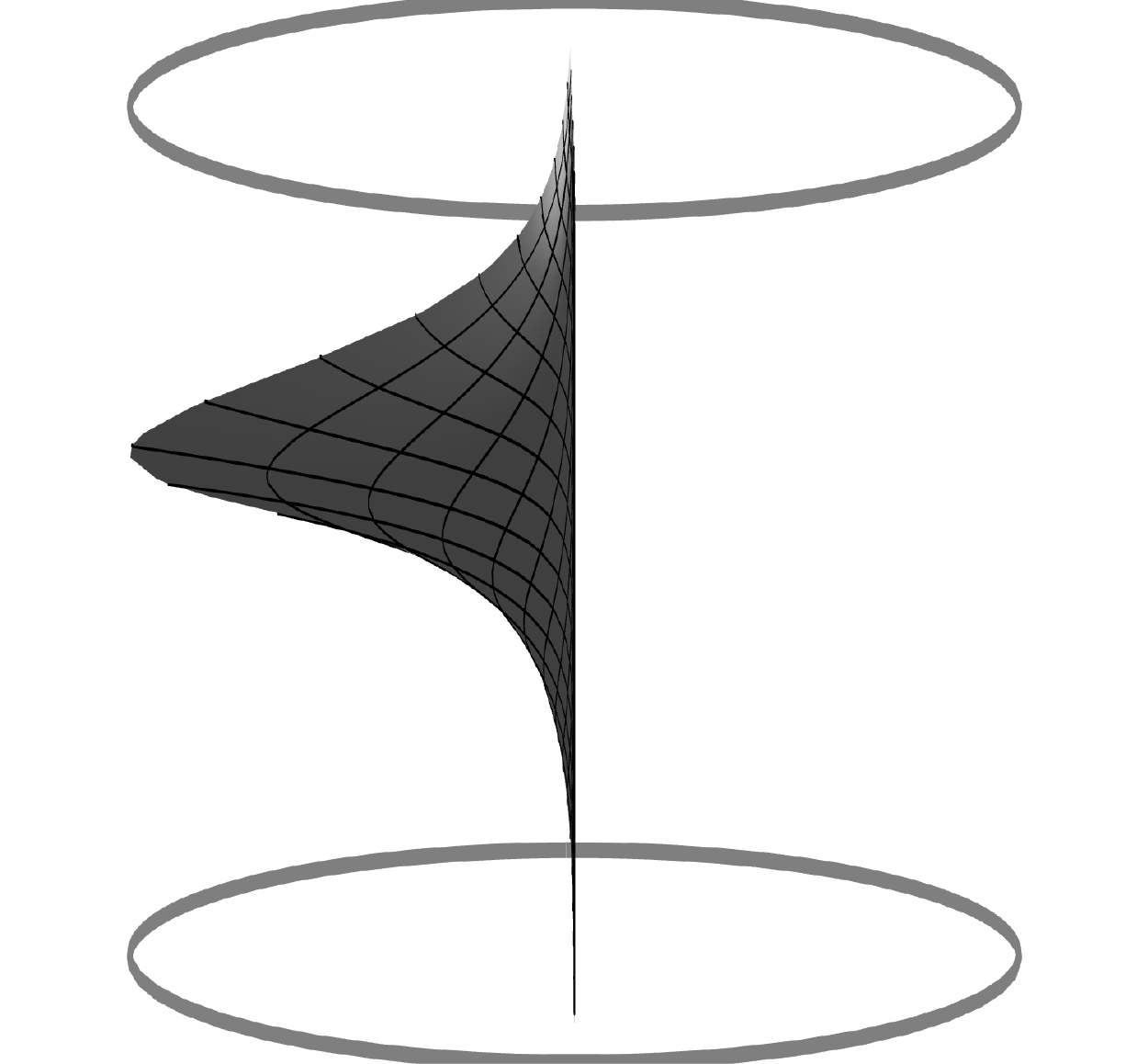} & \includegraphics[width=0.33\textwidth]{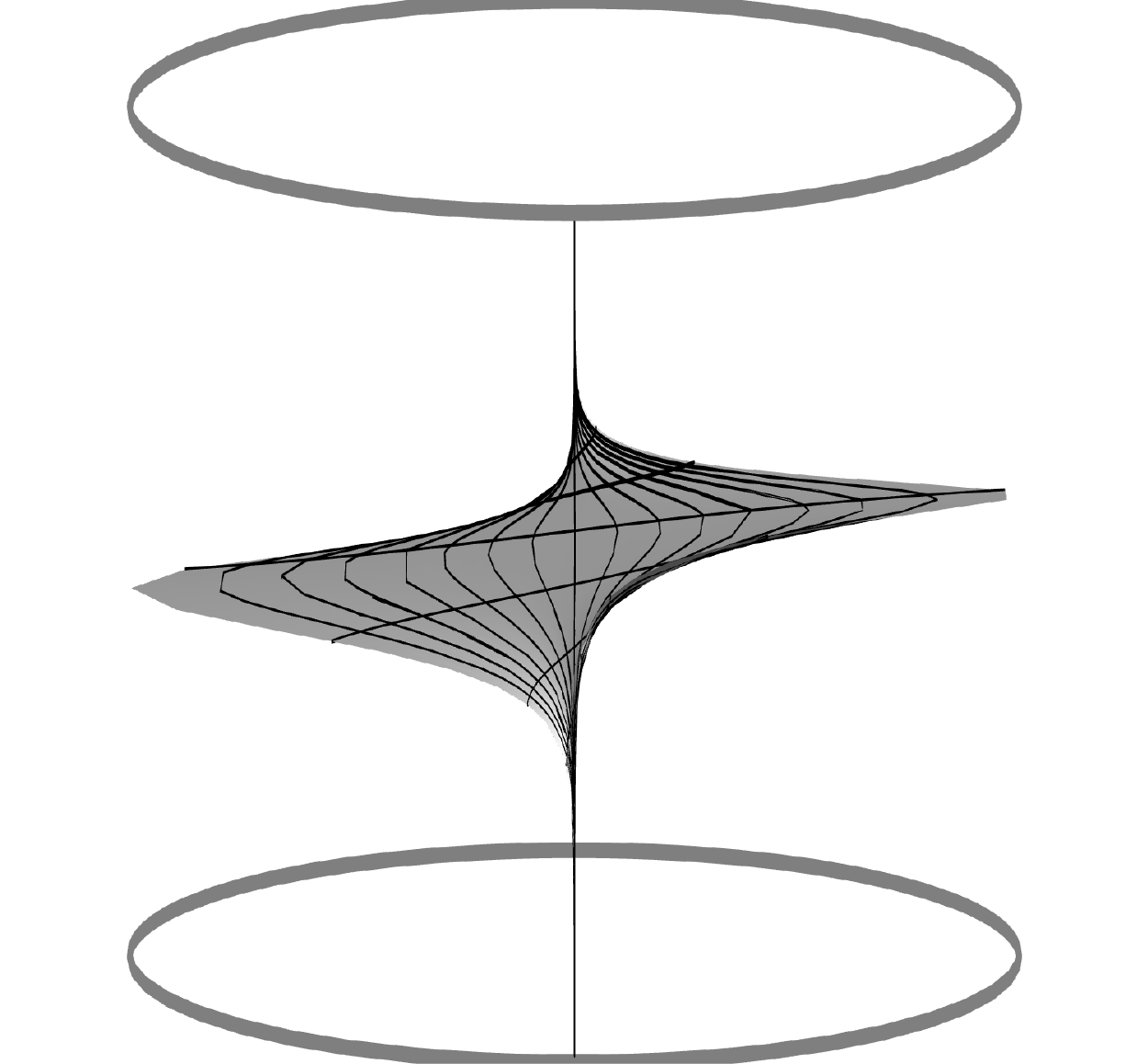} \\
\includegraphics[width=0.33\textwidth]{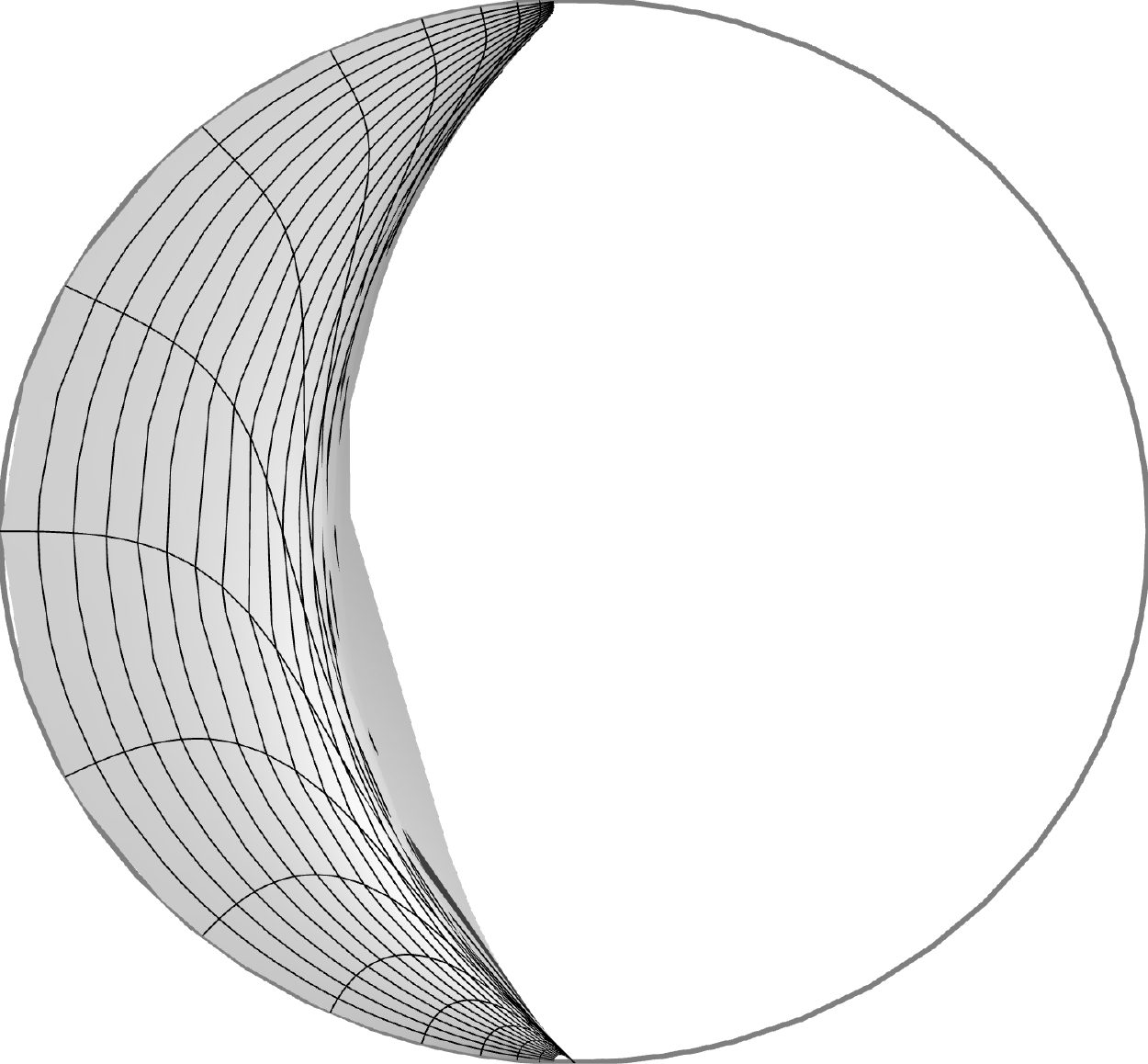}	&  \includegraphics[width=0.33\textwidth]{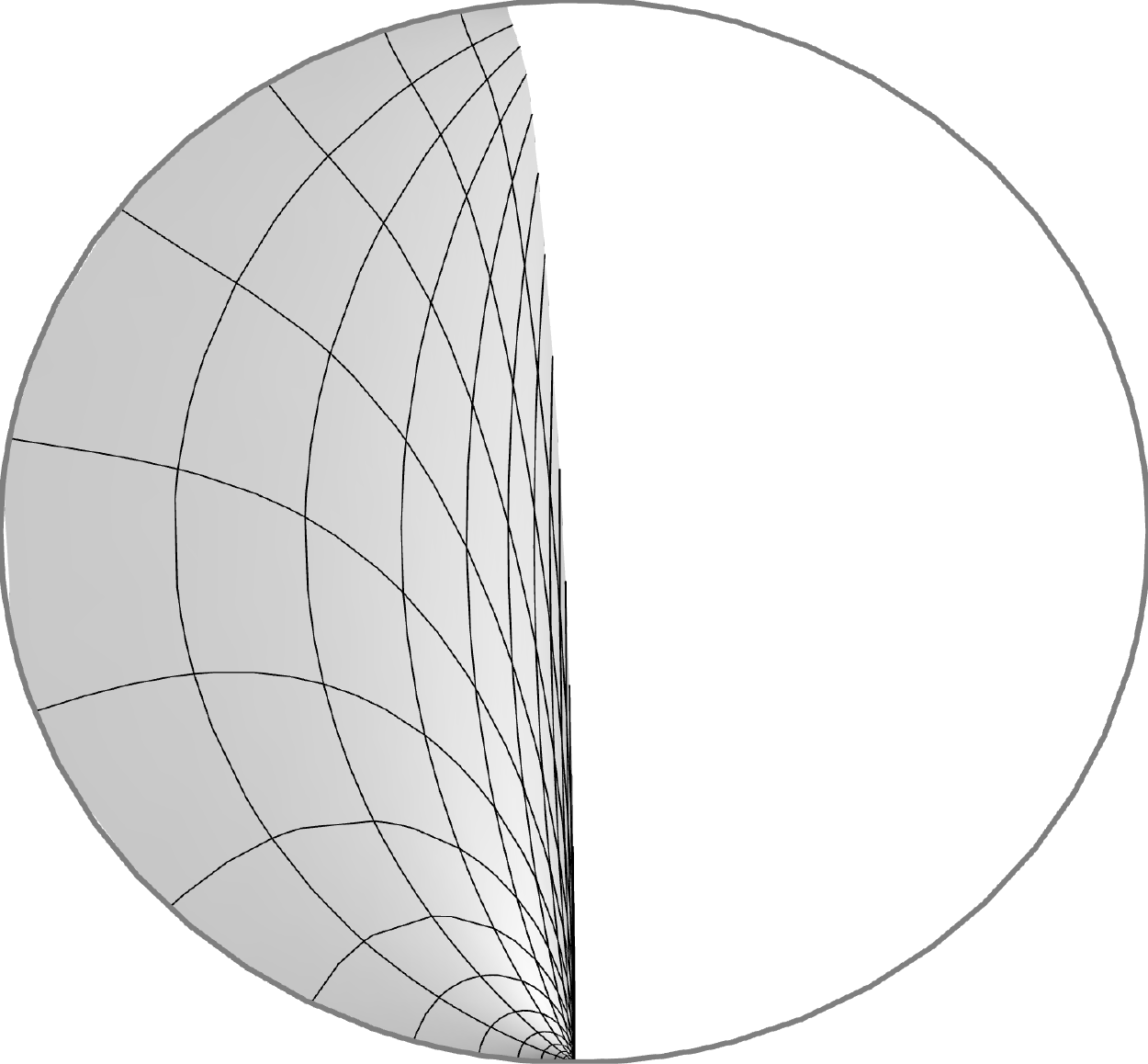} & \includegraphics[width=0.33\textwidth]{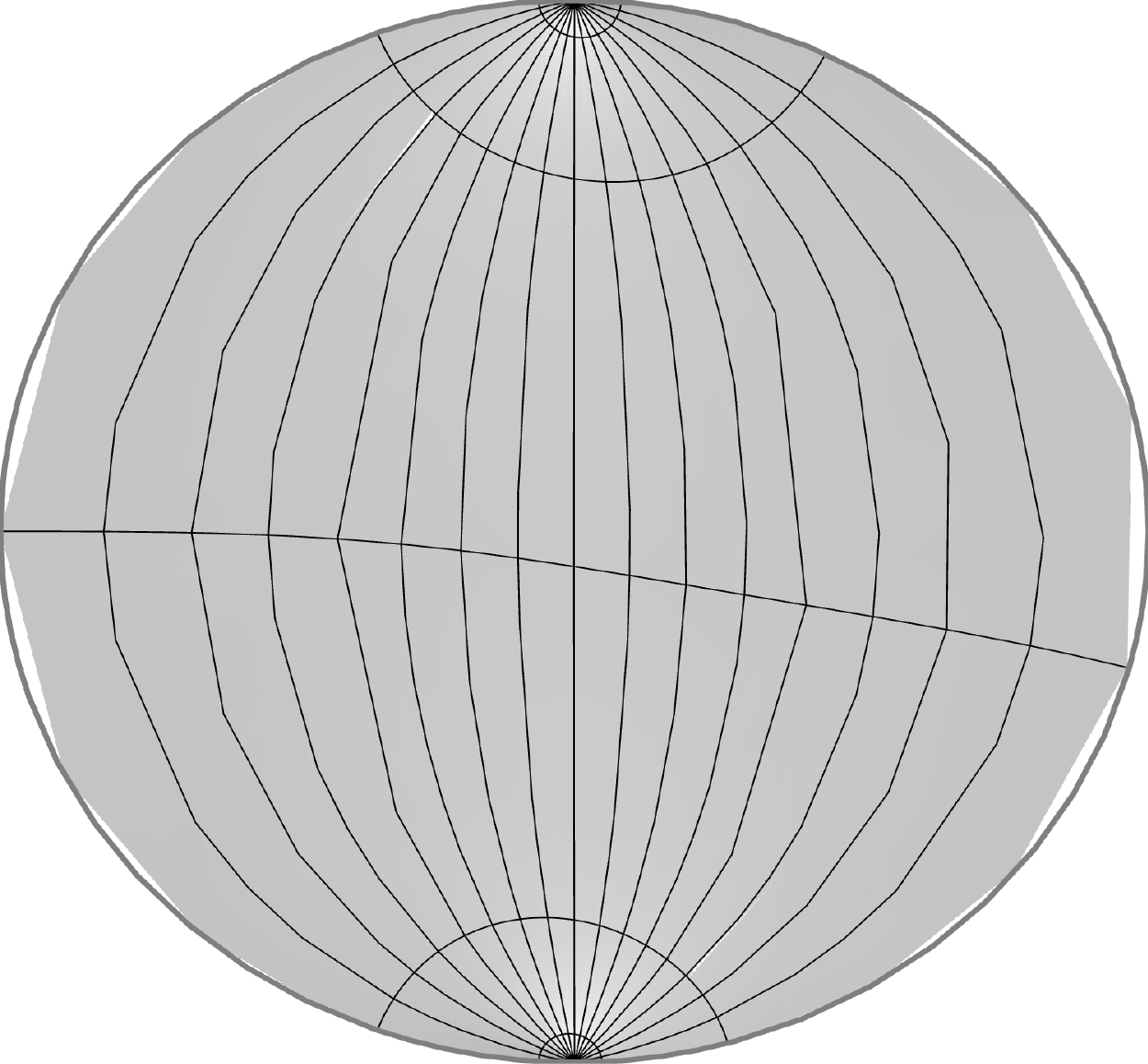}
\end{tabular}
\caption{From left to right, typical solutions for negative energy $E = -0.5, -1, -6$ in $\h^2\times \R$ being $\h^2$ the disc model. Below each surface the top view has been drawn.} \label{fig:negative-energy}
\end{figure}

This can be checked directly from the definition or, taking into account Corollary~\ref{cor:Gauss-map-H2xR}, computing the first component of the Gauss map of the maximal immersion $\phi_E$ given in Proposition~\ref{prop:examples-AdS}.

These examples are invariant by the $1$-parameter group of isometries $A_\theta\times \tau_\theta$ of $\h^2\times \R$, where $A_\theta$ is an isometry of $\h^2$ given by:

\begin{center}
\begin{tabular}{ccc}
	&	$E > 0$ &  $E < 0$ \\ 
	$A_\theta$ & $\begin{pmatrix}
		1 	& 	0 	& 	0 	\\ 	
		0 	&	\cos \theta 	&	-\sin \theta 	\\
		0 	&	\sin \theta 	&	\cos \theta
		\end{pmatrix}$
	& 
	$\begin{pmatrix}
			\cosh\theta 	&	 0 	& 	\sinh\theta	\\
			0 	&	1 	&	0	\\
			\sinh \theta 	&	0 	& 	\cosh \theta
		\end{pmatrix}$
\end{tabular}
\end{center}
and $\tau_\theta:\R \to \R$ is given by $\tau_\theta(t) = t + \tfrac{\theta}{\sqrt{\abs{E}}}$.

The complete classification of constant mean curvature surfaces (in particular the minimal ones) invariant by a $1$-parameter group of $\h^2\times \R$ can be found in~\cite{Onnis2008} and the references therein.

\section{Appendix}\label{sec:appendix}

In this section we will exhibit explicit solutions for the equation $\Delta v - 2\sinh(2v) = 0$. We will restring ourselves to the simplest case, that is, when the function only depends on one variable, i.e.\ $v = v(x)$. In that case it is easy to find a first integral of the equation, namely the energy $E = (v')^2/2 - \cosh(2v)$ is constant for every solution $v$ (cf.~\eqref{eq:sinh-Gordon-ordinary}). Moreover, if $v$ is a solution then $u(x) = -v(x)$ and $w(x) = v(-x)$ are also solutions with the same energy of $v$. Hence, we only need to consider initial conditions $v(0) = v_0 \geq 0$ and $v'(0) = \sqrt{2(E + \cosh(2v_0))} \geq 0$ (note that $E + \cosh (2v(x)) \geq 0$ by the definition of $E$). Thus we are interested in solving the following initial value problem:
\begin{equation}\label{eq:edo-problem}
\begin{split}
v''(x) -2 \sinh(2v(x)) &= 0, \\
v(0) = v_0 \geq 0, \quad  v'(0) &= \sqrt{2(E + \cosh(2v_0))}.
\end{split}
\end{equation}

It is possible to obtain all the solutions to that problem in terms of Elliptic Jacobi functions (see for instance~\cite{BF1971} for further details). Let
\[
F(\varphi, \mu) = \int_0^\varphi \tfrac{\df \theta}{\sqrt{1-\mu\sin^2 \theta}}, \quad 0 \leq \mu \leq 1,
\]
be the \emph{elliptic integral of the first kind with modulus $\mu$}. Then denoting the inverse of $\varphi \mapsto F(\varphi, \mu)$ by $\varphi = \am_\mu(x)$, the elementary Jacobi elliptic functions are given by:
\begin{align*}
\sn_\mu(x) &= \sin \am_\mu(x), &  \dn_\mu(x) &= \sqrt{1 - \mu\sin^2 \am_\mu(x)} \\
\cn_\mu(x) &= \cos \am_\mu(x), & \tn_\mu(x) &= \sn_\mu(x)/\cn_\mu(x)
\end{align*}
The basic properties of these functions are:
\begin{gather*}
	\sn_\mu(x)^2 + \cn_\mu(x)^2 = 1, \quad \mu \sn_\mu(x)^2 + \dn_\mu(x)^2 = 1 \\
	\sn_\mu\bigl(x + 2K(\mu)\bigr) = -\sn_\mu(x), \quad \cn_\mu\bigl(x + 2K(\mu) \bigr) = -\cn_\mu(x), \\
	\dn_\mu(x+2K(\mu)) = \dn_\mu(x), \quad \tn_\mu\bigl( x + 2K(\mu) \bigr) = \tn_\mu(x),
\end{gather*}
where $K(\mu) = F(\tfrac{\pi}{2},\mu)$ is the complete elliptic integral of the first kind. Moreover, the derivaties of the Jacobi elliptic functions are:
\begin{align*}
	\sn_\mu'(x) &= \cn_\mu(x)\dn_\mu(x), & \cn_\mu'(x) &= -\sn_\mu(x) \dn_\mu(x), \\
	\am_m'(x) &= \dn_\mu(x), & \dn_\mu'(x) &= -\mu \sn_\mu(x) \cn_\mu(x).
\end{align*}

\begin{lemma}\label{lm:solutions-sinh-Gordon-ordinary}
The solution $v:I \rightarrow \R$ of the initial value problem~\eqref{eq:edo-problem} and its maximal definition interval $I$ are given, in terms of the energy $E$, by:
\begin{align*}
	[E > 1]\quad v(x) &= \log\bigl(\lambda \tn_\mu(\lambda^{-1}x + a_0))\bigr), & \mu &= 1-\lambda^4,\\
	I &= ]-a_0, \lambda K(\mu)-a_0[, & a_0 &= \arctn_\mu(\lambda^{-1}e^{v_0}),\\
	[\abs{E}\leq 1]\quad  v(x) &= \log \bigl( \tn_\mu(x+a_0)\dn_\mu(x+a_0) \bigr),& \mu &= \tfrac{1-E}{2},\\
	I &= ]-a_0, K(\mu)-a_0[, & a_0 &= \tfrac{1}{2}\arccn_\mu(\tanh(v_0)),\\
	[E < -1]\quad v(x) &= -\log\bigl( \lambda^{-1} \sn_\mu(\lambda x + a_0) \bigr), & \mu &= \lambda^{-4},\\
	I &= ]-a_0, 2\lambda^{-1} K(\mu) - a_0[, & a_0 &= \arcsn_\mu(\lambda e^{-v_0})
\end{align*}
where $\lambda^2 = |E - \sqrt{E^2 - 1}|$ for $\abs{E} > 1$.
\end{lemma}

\begin{proof}
It is a direct computation taking into account the aforementioned properties of the Jacobi elliptic functions.
\end{proof}

\begin{remark}\label{rmk:properties-solutions-sinh-Gordon}

In the special cases $E = 1$ and $E = -1$ we get solutions in terms of elementary functions, namely, $v_1(x) = \log \tan(x)$, $v_{-1}(x) = \log \cotanh(x)$ as well as the constant solution $v(x) = 0$ (also with $E = -1$). 

On the one hand, the solutions of the sinh-Gordon equation with energy $E > -1$ are symmetric with respect to the middle point of the maximal interval of definition. On the other hand, the solutions $v$ with energy $E < -1$ never vanish and are symmetric with respect to the vertical line passing through the middle point of the maximal interval of definition.
\end{remark}

\end{document}